\documentclass[reqno]{amsart}

\usepackage{amsmath}
\usepackage{amsfonts}
\usepackage{amsthm}
\usepackage{amssymb}
\usepackage{graphicx}
\usepackage{graphics}
\usepackage{bm}
\usepackage{dsfont}
\usepackage{color}
\usepackage[font=footnotesize]{caption}
\usepackage[all]{xy}
\numberwithin{equation}{section}
\newtheoremstyle{personal}%
{12pt}%      Space above
{12pt}%      Space below
{\slshape}%         Body font
{}%         Indent amount
{\bfseries}% Theorem head font
{.}%        Punctuation after theorem head
{.5em}%     Space after theorem head
{}%         Theorem head spec (can be left empty, meaning "normal")
\theoremstyle{personal}%
\newtheorem{thm}{Theorem}[section]

\newtheorem{lem}[thm]{Lemma}

\newtheorem{quest}{Question}[section]
\theoremstyle{definition}
\newtheorem{rem}[thm]{Remark}

\newcommand{\N}{\mathds{N}}
\newcommand{\Z}{\mathds{Z}}
\newcommand{\R}{\mathds{R}}

\newcommand{\RP}{\mathds{R}\PP}
\newcommand{\PT}{\PP\Tan}
\newcommand{\PP}{\mathrm{P}}

\newcommand{\EE}{\mathcal{E}}
\newcommand{\diff}{\mathrm{d}}
\newcommand{\Tan}{\mathrm{T}}
\newcommand{\ev}{\mathrm{ev}}
\DeclareMathOperator*{\toup}{\longrightarrow}

\newcommand{\Diff}{\mathrm{Diff}}

\newcommand{\spec}{\sigma_{\mathrm{s}}}

\begin{document}

\title[A characterization of Zoll Riemannian metrics on the 2-sphere]{A characterization of Zoll\\ Riemannian metrics on the 2-sphere}

\author[M. Mazzucchelli]{Marco Mazzucchelli}
\address{Marco Mazzucchelli\newline\indent CNRS, \'Ecole Normale Sup\'erieure de Lyon, UMPA\newline\indent  46 all\'ee d'Italie, 69364 Lyon Cedex 07, France}
\email{marco.mazzucchelli@ens-lyon.fr}

\author[S. Suhr]{Stefan Suhr}
\address{Stefan Suhr\newline\indent Ruhr-Universit\"at Bochum, Fakult\"at f\"ur Mathematik\newline\indent
44780 Bochum, Germany}
\email{stefan.suhr@rub.de}

\date{November 30, 2017. \emph{Revised}: July 20, 2018.}

\begin{abstract}
The simple length spectrum of a Riemannian manifold is the set of lengths of its simple closed geodesics. We prove a theorem claimed by Lusternik: in any Riemannian 2-sphere whose simple length spectrum consists of only one element $\ell$, any geodesic is simple closed with length $\ell$.
\end{abstract}

\maketitle

\section{Introduction}

A remarkable class of closed Riemannian manifolds is given by those all of whose geodesics are closed. A detailed account of the state of the art of the research on this subject up to the late 1970s is contained in the celebrated monograph of Besse \cite{Besse:1978pr}, while for more recent results we refer the reader to, e.g., \cite{Olsen:2010ne, Radeschi:2017dz, Abbondandolo:2017xz} and references therein. The round $n$-spheres are the simplest examples of manifolds in this class. The first non-trivial example of a 2-sphere of revolution all of whose geodesics are closed was given by Zoll \cite{Zoll:1903by}. The closed geodesics in this example are without self-intersections and have the same length. This is not accidental: a theorem of Gromoll and Grove \cite{Gromoll:1981kl} implies that every 2-sphere all of whose geodesics are closed is Zoll, meaning that all the geodesics are simple closed and have the same length. Our main result, which was claimed by Lusternik in \cite[page~82]{Ljusternik:1966tk}, shows that the property of being Zoll for a Riemannian 2-sphere $(S^2,g)$ can be read off from its simple length spectrum $\spec(S^2,g)$, that is, the set of lengths of its simple closed geodesics.
\begin{thm}\label{t:main}
In a Riemannian 2-sphere $(S^2,g)$ such that $\spec(S^2,g)=\{\ell\}$, every geodesic is simple closed and has length $\ell$. 
\end{thm}

Under a weaker assumption on the simple length spectrum, Lusternik also established the following easier statement. We will provide its precise proof in Section~\ref{s:proofs} for the reader's convenience.
\begin{thm}[\cite{Ljusternik:1966tk}, page 81]\label{t:ellipsoid}
Let $(S^2,g)$ be a Riemannian 2-sphere such that $\spec(S^2,g)$ has at most two elements. Then, for some $\ell\in\spec(S^2,g)$, every $x\in S^2$ lies on a simple closed geodesic of $(S^2,g)$ of length $\ell$.
\end{thm}

If one further assumes that the sectional curvature of the Riemannian metric takes values inside $[1/4,1]$, Theorems~\ref{t:main} and~\ref{t:ellipsoid} are a consequence of the following result of Ballmann, Thorbergsson and Ziller \cite[Theorem~A]{Ballmann:1983fv}. Consider a Riemannian $n$-sphere, with $n\geq 2$, whose sectional curvature $\kappa$ satisfies $1/4\leq\delta\leq \kappa \leq 1$. If all the (not necessarily simple) closed geodesics  with length in $[2\pi,2\pi\delta^{-1/2}]$ have at most two different length values, for one such length $\ell$ every point of the $n$-sphere lies on a closed geodesic of length $\ell$. If all the closed geodesics with length in $[2\pi,2\pi\delta^{-1/2}]$ have the same length, then all the geodesics are closed with the same length.

For any $r\in(0,1)$ sufficiently close to $1$, the ellipsoid of revolution 
\[E(r):=\big\{(x,y,z)\in\R^3\ \big|\ x^2+y^2+(z/r)^{2}=1 \big\}\] 
equipped with the Riemannian metric induced by the ambient Euclidean metric of $\R^3$ satisfies the assumptions of Theorem~\ref{t:ellipsoid}, but is not a Zoll 2-sphere. Indeed, the meridians of $E(r)$ are simple closed geodesics of the same length, and the only other simple closed geodesic is the equator, whose length is $1$. Nevertheless, we do not know whether Theorem~\ref{t:ellipsoid} is optimal.

\begin{quest}
On a Riemannian 2-sphere $(S^2,g)$ such that $\spec(S^2,g)$ has at most two elements, is there a length $\ell\in\spec(S^2,g)$ and a point $x\in S^2$ such that every geodesic going through $x$ is simple closed with length $\ell$?
\end{quest}

The proofs of Theorems~\ref{t:main} and~\ref{t:ellipsoid} build on the classical minmax recipe due to Lusternik and Schnirel\-mann \cite{Lusternik:1934km, Ballmann:1978rw} for detecting three simple closed geodesics on every Riemannian 2-sphere. A crucial ingredient for this recipe is a deformation that shrinks emdedded loops without creating self-intersections, which can be obtained by applying Grayson's curve shortening flow \cite{Grayson:1989ec}. We expect such a flow to be available also in the setting of reversible Finsler metrics. If this were the case, Theorems~\ref{t:main} and~\ref{t:ellipsoid} would extend to reversible Finsler metrics on the 2-sphere. 

We close the introduction by raising one more question related to Theorem~\ref{t:main}. Consider the unit tangent bundle $SS^2$ equipped with a contact form $\alpha$. The associated Reeb vector field $R$ on $SS^2$ is defined by $\alpha(R)\equiv1$ and $\diff\alpha(R,\cdot)\equiv0$. We say that $(SS^2,\alpha)$ is reversible when $\phi^* R=-R$, where $\phi:SS^2\to SS^2$ is the involution $\phi(x,v)=(x,-v)$.

\begin{quest}
Assume that all the periodic orbits of the Reeb vector field of a reversible $(SS^2,\alpha)$ have the same period. Is $(SS^2,\alpha)$ a Zoll contact manifold, namely such that all its Reeb orbits are periodic?
\end{quest}

\footnotesize
\subsection*{Acknowledgments.} We are grateful to Alberto Abbondandolo, who asked us the original question leading to Theorem~\ref{t:main}, and suggested to convert our homological proof in a cohomological one, which resulted in a dramatic simplification of the exposition. We also thank Wolfgang Ziller for pointing out to us the above mentioned result in \cite{Ballmann:1983fv}. Marco Mazzucchelli is partially supported by the ANR-13-JS01-0008-01 ``Contact spectral invariants''. Stefan Suhr is supported by the SFB/TRR 191 ``Symplectic Structures in Geometry, Algebra and Dynamics'', funded by the Deutsche Forschungsgemeinschaft. Part of this work was carried out during a visit of Marco Mazzucchelli at the Ruhr-Universit\"at Bochum in November 2017, funded by the SFB/TRR 191; both authors wish to thank the university for providing a wonderful working environment. 

\normalsize

\section{Lusternik-Schnirelmann theory}

In their celebrated work \cite{Lusternik:1934km}, Lusternik and Schnirelmann showed how to detect three simple closed geodesics on every Riemannian 2-sphere by applying variational methods. The original proof of this fact in~\cite{Lusternik:1934km} is known to have a gap. More specifically, the argument requires a deformation of the space of unparametrized embedded circles in the 2-sphere that shrinks all those circles that are not closed geodesics. The deformation provided by Lusternik and Schnirelmann is incomplete. Actually, constructing such a deformation by hand turned out to be highly non-trivial, and several authors proposed their solution in the second half of the 20th century. A particularly elegant one was provided by Grayson \cite{Grayson:1989ec} with its curve shortening flow. In this section, we are going to review the arguments leading to Lusternik-Schnirelmann's theorem in combination with Grayson's work. For the topological arguments, we will mainly follow \cite{Ballmann:1978rw}.

\subsection{Grayson's curve shortening flow}

Let $(S^2,g)$ be an oriented Riemannian 2-sphere. We denote by $\Omega\subset C^\infty(S^1,S^2)$ the space of embedded circles $\gamma:S^1\hookrightarrow S^2$, and by $\Omega_0\subset C^\infty(S^1,S^2)$ the space of constant maps. The group $\Diff(S^1)$ acts on $\Omega_*:=\Omega\cup\Omega_0$ by reparametrizations, i.e.
\begin{align*}
(f\cdot\gamma)(t)=\gamma(f(t)),\qquad
\forall f\in\Diff(S^1),\ \gamma\in\Omega,\ t\in S^1.
\end{align*}
We consider the space of unparametrized loops
$\Lambda_*:=\Omega_*/\Diff(S^1)$ endowed with the quotient Whitney $C^\infty$ topology, and its subsets $\Lambda:=\Omega/\Diff(S^1)$ and $\Lambda_0:=\Omega_0/\Diff(S^1)\equiv\Omega_0\cong S^2$. We also consider the length function
\begin{align*}
L:\Lambda_* \to [0,\infty),
\qquad
L(\gamma)=\int_{S^1} g(\dot\gamma(t),\dot\gamma(t))^{1/2} \diff t,
\end{align*}
which is continuous.

For each parametrized embedded circle $\gamma\in\Omega$, we denote by $\nu_\gamma$ its positive normal vector field, so that the ordered pair $\{\dot\gamma(s),\nu_{\gamma}(s)\}$ is an oriented orthonormal basis of the tangent space $\Tan_{\gamma(s)}S^2$ for each $s\in S^1$. We also denote by $\kappa_{\gamma}:S^1\to\R$ the signed curvature of $\gamma$, which is defined by 
\[\kappa_\gamma(s)=\frac{g(\nabla_s\dot\gamma,\nu_{\gamma}(s))}{\|\dot\gamma(s)\|_g^{2}}.\] 
Up to a sign, both $\nu_\gamma$ and $\kappa_\gamma$ are independent of the parametrization of $\gamma$; more precisely, for all $f\in\Diff(S^1)$, we have
\begin{align*}
\nu_{f\cdot\gamma}=\mathrm{sign}(\dot f)\,\nu_{\gamma}\circ f,
\qquad
\kappa_{f\cdot\gamma}=\mathrm{sign}(\dot f)\, \kappa_{\gamma}\circ f.
\end{align*}
In particular, the product  $\kappa_{\gamma}\nu_{\gamma}$ is completely independent of the parametrization of $\gamma$, that is,
\begin{align}\label{e:invariance}
\kappa_{f\cdot\gamma}(s) \nu_{f\cdot\gamma}(s)= \kappa_{\gamma}(f(s))\,\nu_{\gamma}(f(s)),
\qquad\forall s\in S^1.
\end{align}
Now, let us consider the parabolic partial differential equation 
\begin{equation}
\label{e:curve_shortening}
 \partial_t \gamma_t(s) = \kappa_{\gamma_t}(s)\,\nu_{\gamma_t}(s)
\end{equation}
with initial condition $\gamma_0=\gamma$. By~\eqref{e:invariance}, this equation is parametrization invariant. Namely, if $\gamma_t\in\Omega$ is a solution of 
\eqref{e:curve_shortening}, for each $f\in\Diff(S^1)$ the family of curves $\zeta_t:=f\cdot\gamma_t$ is a solution of the same equation with initial condition $\zeta_0=f\cdot\gamma$. Therefore, we can view~\eqref{e:curve_shortening} as a recipe that prescribes the evolution of unparametrized embedded circles $\gamma\in\Lambda$.

The local existence, uniqueness, and continuous dependence on the initial condition in the $C^\infty$ topology of the solutions of~\eqref{e:curve_shortening} is well known by the standard theory of parabolic partial differential equations  (see, e.g., \cite[Theorem 1.1]{ManteMarti} for a modern account). In his fundamental paper \cite{Grayson:1989ec}, Grayson studied the long-term existence and several properties of the solutions of \eqref{e:curve_shortening}.  Summing up, there is an open neighborhood $J\subset[0,\infty)\times \Lambda$ of $\{0\}\times\Lambda$ and a continuous map $\phi:J\to\Lambda$ encoding the solutions of \eqref{e:curve_shortening}, in the sense that $\phi(t,\gamma)=\phi_t(\gamma):=\gamma_t$. Such map $\phi$ is referred to in the literature as the curve shortening flow, and satisfies the following properties. For each $\gamma\in\Lambda$, we denote by $\tau_\gamma\in(0,\infty]$ the largest extended real number such that $[0,\tau_\gamma)\times \{\gamma\} \subset J$. 
\begin{itemize}
\item[(i)] For all $(t,\gamma)\in J$ we have 
\[\frac{\diff}{\diff t} L(\phi_t(\gamma)) = -\int_{S^1} \kappa_{\gamma_t}(s)^2\|\dot\gamma_t(s)\|_g \diff s \leq0, \] 
with equality if and only if $\gamma$ is a closed geodesic (notice that in the integrand above we have introduced a parametrization of $\gamma_t\in\Lambda$, but the value of the integral is independent of this choice); see \cite[page 75]{Grayson:1989ec}.

\item[(ii)] For each $\gamma\in \Lambda$, the limit
\begin{align*}
\ell_{\gamma}:=\lim_{t\to\tau_\gamma} L(\phi_t(\gamma))
\end{align*}
exists; if $\tau_\gamma<\infty$ then $\ell_\gamma=0$ and $\phi_t(\gamma)$ converges to some constant curve in $\Lambda_0$ as $t\to\tau_\gamma$; otherwise, $\ell_\gamma>0$ and, for each open neighborhood $U\subset\Lambda$ of the set of simple closed geodesics of length $\ell_\gamma$ and for all $t>0$ large enough, $\phi_t(\gamma)$ belongs to $U$; see \cite[Theorem 0.1]{Grayson:1989ec}.

\item[(iii)] Let $U\subset\Lambda$ be an open neighborhood of the subspace of simple closed geodesics of length $\ell$; there exists $\epsilon=\epsilon(U)>0$ such that, for every compact subset $K\subset\{L\leq \ell+\epsilon\}$, there exists a continuous function $\tau:K\to[0,\infty)$ such that $\phi_{\tau(\gamma)}(\gamma)\in\{L<\ell\}\cup U$ for all $\gamma\in K$; see \cite[Lemma 8.1]{Grayson:1989ec}.
\end{itemize}

\subsection{The fundamental group of $\Lambda_*$}

Let us construct a 2-fold covering map \[\pi:C\to \Lambda_*.\] The idea of this construction goes as follows. Above the subspace of embedded circles $\Lambda\subset\Lambda_*$, the total space $\pi^{-1}(\Lambda)$ is precisely the space of embedded compact disks in $S^2$, and the projection $\pi$ sends a compact disk to its boundary curve; above any constant $\gamma\in\Lambda_0$, one element of $\pi^{-1}(\gamma)$ must be thought as the collapsed disk at the point $\gamma$, whereas the other element must be thought as the compact disk that fills $S^2$ and whose boundary has been collapsed to $\gamma$. Let us now provide the formal construction of this covering space.

For each $\gamma\in\Lambda_*$ we define a set of two elements $C_\gamma$ as follows: if $\gamma\in\Lambda$, we define $C_\gamma:=\pi_0(S^2\setminus\gamma)$ to be the set of path-connected components of its complement; otherwise, if $\gamma\in\Lambda_0$, we simply set $C_\gamma:=\Z_2=\Z/2\Z$. We set 
\[C:=\big\{(\gamma,Q)\ \big|\ \gamma\in\Lambda_*,\ Q\in C_\gamma\big\}.\] 
We endow $C$ with a topology, by defining a fundamental system of open neighborhoods of any point $(\gamma,Q)\in C$ as follows. 
\begin{itemize}
\item Assume first that $\gamma\in\Lambda$, so that $Q$ is a connected component of $S^2\setminus\gamma$, and choose an arbitrary point $x\in Q$. Let $U\subset\Lambda$ be a sufficiently small open neighborhood of $\gamma$ so that, for each $\zeta\in U$,  $x$ does not lie on the curve $\zeta$. For every such $\zeta$, we set $Q_\zeta\in\pi_0(S^2\setminus\zeta)$ to be the connected component containing $x$.

\item Assume now that $\gamma\in\Lambda_0$, so that $Q\in\Z_2$. We choose an arbitrary point $x\in S^2\setminus\gamma$, and as before a sufficiently small open neighborhood $U\subset\Lambda_*$ of $\gamma$ such that, for every $\zeta\in U$, $x$ does not lie on  $\zeta$. For each $\zeta\in U\cap\Lambda_0$, we set $Q_\zeta:=Q$. For each $\zeta\in U\cap\Lambda$, if $Q=1$ we set $Q_\zeta\in\pi_0(S^2\setminus\zeta)$ to be the  connected component containing $x$, whereas if $Q=0$ we set $Q_\zeta$ to be the connected component not containing $x$.
\end{itemize}

In both cases, we declare $U':=\big\{ (\zeta,Q_\zeta)\ |\ \zeta\in U \}\subset C$ to be an open neighborhood of $(\gamma,Q)$. With this topology, $C\to \Lambda_*$ is a 2-fold covering map with projection $(\gamma,Q)\mapsto\gamma$.

Let $\gamma_0\in\Lambda_0$ be a constant curve. We employ the covering $C$ to define a group homomorphism 
\begin{align}\label{e:A}
A : \pi_1(\Lambda_*,\gamma_0) \to \Z_2 
\end{align}
as follows. Consider a continuous loop $\Gamma:[0,1]\to\Lambda_*$ with $\Gamma(0)=\Gamma(1)=\gamma_0$. We lift $\Gamma$ to a continuous path $\widetilde\Gamma:[0,1]\to C$ such that $\widetilde\Gamma(0)=(\Gamma(0),0)$ and the following diagram commutes
\begin{align*}
 \xymatrix{
      & C \ar[d] \\
    [0,1] \ar[ur]^{\widetilde\Gamma}\ar[r]^{\Gamma} & \Lambda_*
  }
\end{align*}
We write $\widetilde\Gamma(t)=:(\Gamma(t),Q(t))$, and set $A([\Gamma]):=Q(1)$. The fact that $A([\Gamma])$ only depends on the homotopy class $[\Gamma]\in\pi_1(\Lambda_*,\gamma_0)$ readily follows from the homotopy lifting property of the covering $C\to\Lambda_*$. The homomorphism property \[A([\Gamma']*[\Gamma''])=A([\Gamma'])+A([\Gamma''])\] is a consequence of the following observation: if $\widetilde\Gamma_0:[0,1]\to C$ and $\widetilde\Gamma_1:[0,1]\to C$ are the two lifts of a continuous loop $\Gamma:S^1\to\Lambda_*$ as above with $\widetilde\Gamma_0(0)=(\Gamma(0),0)$ and $\widetilde\Gamma_1(0)=(\Gamma(0),1)$, then if we write $\widetilde\Gamma_i(t)=:(\Gamma(t),Q_i(t))$ we have $Q_0(1)=1-Q_1(1)$. It is not hard to see that the homomorphism $A$ is surjective (see the proof of  Lemma~\ref{l:pi_1_injective}).

\subsection{Three subordinate homology classes}
As usual, we denote by $B^{n+1}$ the closed unit ball in $\R^{n+1}$, by $S^{n}=\partial B^{n+1}$ the unit $n$-sphere, and by $\RP^n=S^n/\sim$ the $n$-dimensional real projective space, where $x\sim-x$ for each $x\in S^n$. We will always identify $\R^{n-1}$ with the hyperplane $\R^{n-1}\times\{0\}\subset\R^n$, so that in the sequence $S^0\subset S^1\subset S^2$ each sphere is the equator of the next one, and analogously we have the sequence of inclusions $\RP^0\subset\RP^1\subset \RP^2$. We denote by $E$ the space
\begin{align*}
E := \big\{ ([x],\lambda x)\in\RP^2\times B^{3}\ \big|\ x\in S^2,\ \lambda\in[-1,1]\big\},
\end{align*}
which is the total space of the tautological unit-ball bundle over $\RP^2$ with projection
$p:E\to\RP^2$, $p([x],\lambda x)=[x]$.
We recall that the cohomology ring of $E$ with coefficients in $\Z_2$ is given by $H^*(E;\Z_2)=\Z_2[\omega]/(\omega^{3})$, where $\omega$ is the generator of $H^1(E;\Z_2)\cong\Z_2$. Let $\tau\in H^1(E,\partial E;\Z_2)$ be the Thom class of the tautological bundle, which gives the Thom isomorphism
\begin{align*}
H^j(E;\Z_2)  \toup^{\cong} \,& H^{j+1}(E,\partial E;\Z_2),\\
\alpha  \longmapsto \,& \tau\smallsmile\alpha.
\end{align*}
In particular, $H^{*}(E,\partial E;\Z_2)$ is generated as a group by the cohomology classes $\tau\smallsmile\omega^j$, for $j=0,1,2$. Let $h_3$ be the (non-zero) homology class in $H_3(E,\partial E;\Z_2)$ such that $(\tau\smallsmile\omega^2)(h_3)=1$. We define the non-zero homology classes 
\begin{align*}
h_{2}:=\omega\smallfrown h_{3}\in H_2(E,\partial E;\Z_2),\\ 
h_{1}:=\omega\smallfrown h_{2}\in H_1(E,\partial E;\Z_2).
\end{align*}
We now define a map 
\begin{align}\label{e:iota}
\iota: (E,\partial E)\to (\Lambda_*,\Lambda_0) 
\end{align}
as follows: for each $([x],v)\in E$, the (possibly constant) loop $\iota([x],v)$ is given  by the intersection of $S^2$ with the affine plane 
$P([x],v):=\mathrm{span}\{x\}^\bot + v \subset \R^{3}$,
see Figure~\ref{f:iota}.
\begin{figure}
\begin{center}
\begin{small}
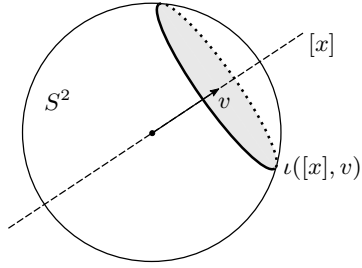 
\caption{The map $\iota:(E,\partial E)\to(\Lambda_*,\Lambda_0)$.}
\label{f:iota}
\end{small}
\end{center}
\end{figure}

\begin{lem}\label{l:pi_1_injective}
For any $e_0\in\partial E$ with image $\gamma_0:=\iota(e_0)$, the map $\iota$ induces  injective homomorphisms
$\iota_*  :\pi_1(E,e_0)\hookrightarrow\pi_1(\Lambda_*,\gamma_0)$ and $\iota_* :H_1(E;\Z_2)\hookrightarrow H_1(\Lambda_*;\Z_2)$.
\end{lem}

\begin{proof}
Let $\Gamma:S^1\to E$ be a continuous loop such that $\Gamma(0)=e_0$ and the homotopy class $[\Gamma]$ generates the fundamental group $\pi_1(E,e_0)$. For instance we can define $\Gamma$ as follows. Let $\Psi:S^1\to\RP^1\subset \RP^2$ be a continuous map of degree 1. We can see $\Psi$ as a loop in the 0-section of $E$, so that is represents a generator of the fundamental group of $\pi_1(E,\Psi(0))$. Let $\Theta:[0,1]\to E$ be a continuous path that joins $e_0$ with $\Psi(0)$. We can set $\Gamma$ to be the loop obtained by the concatenation $\Theta * \Psi * \overline{\Theta}$. Notice that $\iota\circ\Psi$ is the loop of meridians of the sphere $S^2$ that starts at a meridian and applies to it a rotation of angle $\pi$ around the axis passing through the north and south poles. We readily see that $A\circ\iota_*([\Gamma])=1$, where $A$ is the homomorphism in \eqref{e:A}. In particular, $\iota_*:\pi_1(E,e_0)\to\pi_1(\Lambda_*,\gamma_0)$ is non-trivial. Since $\pi_1(E,e_0)\cong\Z_2$, $\iota_*$ is injective and the composition $A\circ\iota_*$ is an isomorphism. Moreover, $\iota_*([\Gamma])$ does not belong to the commutator subgroup $[\pi_1(\Lambda_*,\gamma_0),\pi_1(\Lambda_*,\gamma_0)]$, for otherwise it would belong to the kernel of $A$. Therefore, by Hurewicz Theorem, $\iota\circ\Gamma$ represents a non-zero element of the homology group $H_1(\Lambda_*;\Z_2)$, and therefore the homology homomorphism $\iota_*:H_1(E;\Z_2)\hookrightarrow H_1(\Lambda_*;\Z_2)$ is injective.
\end{proof}

\begin{lem}\label{l:non_triviality}
The map $\iota$ induces an injective homomorphism
\begin{align*}
\iota_* & : H_1(E,\partial E;\Z_2)\hookrightarrow H_1(\Lambda_*,\Lambda_0;\Z_2).
\end{align*}
\end{lem}

\begin{proof}
Since $\partial E$ and $\Lambda_0$ are homeomorphic to $S^2$, the long exact sequences of the pairs $(E,\partial E)$ and $(\Lambda_*,\Lambda_0)$ imply that the inclusions induce isomorphisms 
\begin{align*}
H_1(E;\Z_2)\toup^{\cong} H_1(E,\partial E;\Z_2),
\qquad
H_1(\Lambda_*;\Z_2)\toup^{\cong} H_1(\Lambda_*,\Lambda_0;\Z_2). 
\end{align*}
Lemma~\ref{l:pi_1_injective}, together with the commutative diagram
\begin{align*}
  \xymatrix{
     H_1(E;\Z_2)\Big. \ar[r]^{\cong\ \ \ } \ar@{^{(}->}[d]^{\iota_*}& H_1(E,\partial E;\Z_2) \ar[d]^{\iota_*} \\
    H_1(\Lambda_*;\Z_2) \ar[r]^{\cong\ \ \ } & H_1(\Lambda_*,\Lambda_0;\Z_2)
  }
\end{align*}
implies the injectivity of $\iota_*  : H_1(E,\partial E;\Z_2)\hookrightarrow H_1(\Lambda_*,\Lambda_0;\Z_2)$.
\end{proof}

Consider again the generator $\omega$ of $H^1(E;\Z_2)\cong\Z_2$. Since the homomorphism $\iota_*:H_1(E;\Z_2)\hookrightarrow H_1(\Lambda_*;\Z_2)$ is injective, there exists a cohomology class 
\begin{align}\label{e:kappa}
\kappa\in H^1(\Lambda_*;\Z_2) 
\end{align}
such that $\iota^*\kappa=\omega$. This implies that
\begin{equation}\label{e:cap_product_relations}
\begin{split}
\iota_*h_1&=\iota_*( (\iota^*\kappa)\smallfrown h_2) = \kappa\smallfrown\iota_*h_2,\\
\iota_*h_2&=\iota_*( (\iota^*\kappa)\smallfrown h_3) = \kappa\smallfrown\iota_*h_3,
\end{split} 
\end{equation}
and since $\iota_*h_1$ is nonzero, the homology classes $\iota_*h_2$ and $\iota_*h_3$ are non-trivial in $H_*(\Lambda_*,\Lambda_0;\Z_2)$ as well.

\subsection{Lusternik-Schnirelmann minmax values}

For each value $\ell\geq0$, we denote by $\{L\leq\ell\}\subset\Lambda_*$ the corresponding sublevel set of the length functional, and by  $j^\ell:(\{L\leq\ell\},\Lambda_0)\hookrightarrow(\Lambda_*,\Lambda_0)$ the inclusion map. 
The so-called Lusternik-Schnirelmann minmax values are given by
\begin{align*}
\ell_i:= \inf \big\{\ell>0\ \big|\ \iota_*h_i\in j^\ell_* \big(H_i(\{L\leq\ell\},\Lambda_0)\big) \big\},\qquad i=1,2,3.
\end{align*}

\begin{lem}
$\ell_1>0$.
\end{lem}

\begin{proof}
Let us assume by contradiction that $\ell_1=0$. 
We denote by $\rho$ the injectivity radius of $(S^2,g)$, and we fix a constant $\epsilon\in(0,\rho/3)$. 
Since $\ell_1=0$, we can find a relative 1-cycle $\sigma$ representing $\iota_*h_1$ whose support is contained in the sublevel set $\{L<\epsilon\}$. In other words, $\sigma$ is the formal sum $\sigma_1+...+\sigma_n$ of singular 1-simplexes of the form $\sigma_i:[0,1]\to\{L<\epsilon\}$. Let $k\in\N$ be large enough so that, for all $i=1,...,n$ and $t_1,t_2\in[0,1]$ with $|t_1-t_2|\leq1/k$, we have 
\begin{align*}
\min\big\{d(x_1,x_2)\ |\ x_1\in\sigma_i(t_1),\ x_2\in\sigma_i(t_2) \big\}\leq\epsilon.
\end{align*}
Here, $d:S^2\times S^2\to[0,\infty)$ denotes the Riemannian distance function induced by $g$. For each $i=1,...,n$ and $j=0,...,k$, we choose a point $x_{i,j}\in \sigma_i(j/k)$. We should make these choices coherently, in the sense that $x_{i_1,j_1}=x_{i_1,j_2}$ whenever $\sigma_{i_1}(j_1/k)=\sigma_{i_2}(j_2/k)$. Since $L(\sigma_i(j/k))<\epsilon$, we have $d(x_{i,j},x)\leq\epsilon/2$ for all $x\in\sigma_i(j/k)$. Notice that $d(x_{i,j},x_{i,j+1})\leq2\epsilon$. For $i=1,...,n$, we define a continuous map $\gamma_i:[0,1]\to S^2$ such that each restriction $\gamma_i|_{[j/k,(j+1)/k]}$ is a geodesic joining $\gamma_i(j/k)=x_{i,j}$ and $\gamma_i((j+1)/k)=x_{i,j+1}$. Notice that
\begin{align*}
\max \big\{ d(\gamma_i(t),x)\ \big|\ i=1,...,n,\ t\in[0,1],\ x\in\sigma_i(t) \big\}\leq 3\epsilon < \rho.
\end{align*}
For each $i=1,...,n$, we define the continuous homotopy $\sigma_{i,s}:[0,1]\to\Lambda_*$, $s\in[0,1]$, by 
$\sigma_{i,s}(t):=\exp_{\gamma_i(t)}\big( (1-s) \exp_{\gamma_i(t)}^{-1}(\sigma_i(t)) \big)$.
The relative $1$-cycle $\sigma':=\sigma_{1,1}+...+\sigma_{n,1}$ still represents $\iota_*h_1\in H_1(\Lambda_*,\Lambda_0)$, and its support is contained in $\Lambda_0$. Therefore $\iota_*h_1=0$, which contradicts Lemma~\ref{l:non_triviality}.
\end{proof}

\begin{lem}
$\ell_1\leq \ell_2\leq \ell_3$.
\end{lem}

\begin{proof}
This is a direct consequence of the cap product relations~\eqref{e:cap_product_relations}. Indeed, such relations imply that, for any relative cycle $\sigma$ representing $\iota_*h_3$ and for any cocycle $\lambda$ representing $\kappa$, the relative cycles $\lambda\smallfrown \sigma$ and $\lambda^2\smallfrown \sigma$ represent $\iota_*h_2$ and $\iota_*h_1$ respectively, and the supports of these two relative cycles are contained in the support of $\sigma$.
\end{proof}

\begin{thm}[Lusternik-Schnirelmann \cite{Lusternik:1934km}]\label{t:LS}
If $\ell_1=\ell_2$ or $\ell_2=\ell_3$, then for every open neighborhood $U\subset\Lambda_*$ of the set of simple closed geodesics with length $\ell_2$, the cohomology class $\kappa$ restricts to a non-zero cohomology class in $H^1(U;\Z_2)$. If furthermore $\ell_1=\ell_2=\ell_3$, then also the cohomology class $\kappa^2$ restricts to a non-zero cohomology class in $H^2(U;\Z_2)$.
\end{thm}

\begin{proof}
Assume that $\ell:=\ell_i=\ell_{i+1}$ for some $i\in\{1,2\}$, and let $U\subset\Lambda_*$ be an open neighborhood of the set of simple closed geodesics with length $\ell$. Let $\epsilon=\epsilon(U)>0$ be the constant given by property~(iii) of the curve shortening flow. By the definition of the minmax value $\ell_{i+1}$ there exists a relative cycle $\sigma$ that represents $\iota_*h_{i+1}$ and has support $K$ contained inside the sublevel set $\{L\leq \ell+\epsilon\}$. Since $K$ is compact, by property~(iii) of the curve shortening flow $\phi_t$ there exists a continuous function $\tau:K\to[0,\infty)$ such that the image of the map $\Phi:K\to\Lambda_*$, $\Phi(\gamma)=\phi_{\tau(\gamma)}(\gamma)$ is contained in $\{L<\ell\}\cup U$. Clearly, $[\Phi_*\sigma]=[\sigma]$ in $H_1(\Lambda_*,\Lambda_0;\Z_2)$. After applying sufficiently many times a barycentric subdivision to all the singular simplexes in $\Phi_*\sigma$, we obtain that each singular simplex in $\Phi_*\sigma$ is contained in $\{L<\ell\}$ or in $U$. In particular, we have $\Phi_*\sigma=\sigma'+\sigma''$, where $\sigma'$ and $\sigma''$ are singular chains (but not necessarily cycles) whose supports are contained in $\{L<c\}$ and $U$ respectively.
Now, assume by contradiction that the restriction of the cohomology class $\kappa$ to $U$ is trivial. The cohomology long exact sequence of the pair $(\Lambda_*,U)$ implies that $\kappa$ belongs to the image of the homomorphism $H^1(\Lambda_*,U;\Z_2)\to H^1(\Lambda_*;\Z_2)$ induced by the inclusion. Namely, $\kappa$ can be represented by a cocycle $\lambda$ whose kernel contains all the singular chains with support in $U$. In particular
\begin{align*}
\iota_*h_i
=\kappa\smallfrown \iota_*h_{i+1}
=[\lambda\smallfrown\sigma' + \lambda\smallfrown\sigma'']
=[\lambda\smallfrown\sigma'].
\end{align*}
But this would imply that $\iota_*h_i$ is represented by the relative cycle $\lambda\smallfrown\sigma'$ whose support is contained in the sublevel set $\{L<\ell\}$. This contradicts the fact that $\ell=\ell_i$ is the minmax value associated to $\iota_*h_i$.
The assertion concerning the case where $\ell_1=\ell_2=\ell_3$ follows by an analogous argument.
\end{proof}

\begin{rem}
As we mentioned in the introduction, we expect a curve shortening flow to be available also for reversible Finsler metrics on $S^2$, and thus Theorem~\ref{t:LS}, as well as Theorems~\ref{t:main} and~\ref{t:ellipsoid}, should extend to the reversible Finsler setting.
\end{rem}

\section{Proofs of the theorems}
\label{s:proofs}

Theorem~\ref{t:main} is a consequence of the following statement.
\begin{thm}
Let $(S^2,g)$ be a Riemannian 2-sphere whose Lusternik-Schnirelmann minmax values satisfy $\ell_1=\ell_2=\ell_3=:\ell$. Then, every geodesic of $(S^2,g)$ is a simple closed geodesic of length $\ell$.
\end{thm}

\begin{proof}
We define
$\EE
:=
\big\{
(\gamma,x)\in\Lambda\times S^2\ \big|\ x\in\gamma
\big\}$, which is the total space of the circle bundle $\pi:\EE\to\Lambda$, $\pi(\gamma,x)=\gamma$. We denote by $\PT S^2$ the projectivization of the tangent bundle of $S^2$, that is, $\PT S^2:= (\Tan S^2\setminus\{\mbox{0-section}\})/\sim$,
where $\sim$ is the equivalence relation $(x,v)\sim(x,\lambda v)$ for each real number $\lambda\neq0$. We have a continuous evaluation map 
\begin{align*}
\ev:\EE\to\PT S^2,
\qquad
\ev(\gamma,x)=\Tan_x\gamma.
\end{align*}
Let $G\subset\Lambda$ be the set of great circles in $S^2$. Notice that, if $\iota:E\to\Lambda$ is the map introduced in~\eqref{e:iota}, we have $G=\iota(\{\mbox{0-section}\})\cong\RP^2$. In particular, if  $\kappa\in H^1(\Lambda_*;\Z_2)$ is the cohomology class of~\eqref{e:kappa}, the restriction $\kappa^2|_{G}$ is a generator of $H^2(G;\Z_2)\cong\Z_2$. The Gysin sequence of the restricted circle bundle $\pi:\EE|_G\to G$ reads
\begin{align*}
\xymatrix@R=8pt@C-6pt{
  H^3(G;\Z_2) \ar[r]^{\pi^*} \ar@{=}[d] &
  H^3(\EE|_G;\Z_2) \ar[r]^{\pi_*} &
  H^2(G;\Z_2) \ar[r] \ar@{=}[d] &
  H^4(G;\Z_2) \ar@{=}[d]\\
  0 &  & \Z_2 & 0 
}
\end{align*}
and therefore we have an isomorphism $\displaystyle\pi_*:H^3(\EE|_G;\Z_2)\toup^{\cong} H^2(G;\Z_2)$. The evaluation map restricts to a homeomorphism $\ev|_{\EE|_{G}}:\EE|_{G}\to\PT S^2$. Therefore $\kappa^2|_{G}=\pi_*\ev|_{\EE|_{G}}^*\nu$, where $\nu$ is the generator of $H^3(\PT S^2;\Z_2)$. This, together with the commutative diagram
\begin{align*}
\xymatrix{
  \EE|_G\, \ar@{^{(}->}[r] \ar[d]^{\pi} &
  \EE \ar[r]^{\ev\ \ \ \ } \ar[d]^{\pi} &
  \PT S^2\\
  G\, \ar@{^{(}->}[r] & \Lambda  &  
}
\end{align*}
readily implies that $\kappa^2=\pi_*\ev^*\nu\neq 0$ in $H^2(\Lambda;\Z_2)$.

Now, assume by contradiction that there exists $y=(x,[v])\in\PT S^2$ such that no simple closed geodesic of $(S^2,g)$ is tangent to $v$. The subset 
\begin{align*}
U:=\pi(\ev^{-1}(\PT S^2\setminus\{y\}))
\end{align*}
is therefore an open neighborhood of the set of simple closed geodesics with length $\ell$. We denote by $j:U\hookrightarrow \Lambda$ and $\widetilde j:\EE|_U\hookrightarrow \EE$ the inclusions, so that we have the commutative diagram
\begin{equation}\label{e:final_diagram}
\begin{split}
\xymatrix{
  \EE|_U\, \ar@{^{(}->}[r]^{\widetilde j} \ar[d]^{\pi} &
  \EE \ar[r]^{\ev\ \ \ \ } \ar[d]^{\pi} &
  \PT S^2\\
  U\, \ar@{^{(}->}[r]^j & \Lambda  &  
} 
\end{split}
\end{equation}
Notice that $\ev\circ\widetilde j(\EE|_U)=\PT S^2\setminus\{y\}$. Since $\ell=\ell_1=\ell_2=\ell_3$, Theorem~\ref{t:LS} implies that $j^*\kappa^2\neq 0$ in $H^2(U;\Z_2)$. This, together with~\eqref{e:final_diagram}, implies that
\begin{align*}
0\neq j^*\kappa^2 = j^*\pi_*\ev^*\nu = \pi_*\widetilde{j}^*\ev^*\nu = \pi_*(\ev\circ\widetilde{j})^*\nu,
\end{align*}
and therefore $(\ev\circ\widetilde{j})^*\nu\neq0$ in $H^3(\EE|_{U};\Z_2)$. In particular, the homology homomorphism  $(\ev\circ\widetilde{j})_*:H_3(\EE|_U)\to H_3(\PT S^2;\Z_2)$ is surjective, and we conclude that the map $\ev\circ\widetilde{j}$ must be surjective as well, which is a contradiction.
\end{proof}

Theorem~\ref{t:ellipsoid} is a consequence of the following statement.

\begin{thm}
Let $(S^2,g)$ be a Riemannian 2-sphere whose Lusternik-Schnirelmann minmax values satisfy $\ell_1=\ell_2$ or $\ell_2=\ell_3$. Then every $x\in S^2$ lies on a simple closed geodesic of $(S^2,g)$ of length $\ell_2$.
\end{thm}
\begin{proof}
Assume by contradiction that $\ell_1=\ell_2$ or $\ell_2=\ell_3$, but there exists $x\in S^2$ such that no simple closed geodesic of length $\ell_2$ passes through $x$. The subset
$U:=\big\{\gamma\in\Lambda_*\ \big|\ x\not\in\gamma\big\}$
is therefore an open neighborhood of the set of simple closed geodesics of length $\ell_2$. We claim that $U$ is contractible. Indeed, consider a smooth deformation retraction $r_t:S^2\setminus\{x\}\to S^2\setminus\{x\}$ such that $r_0=\mathrm{id}$, $r_t$ is an embedding for each $t\in[0,1)$, and $r_1(y)=-x$ for all $y\in S^2\setminus\{x\}$. This induces a deformation retraction $R_t:U\to U$, $R_t(\gamma)=r_t(\gamma)$, whose time-1 map $R_1$ contracts $U$ at the constant loop at $-x$. In particular, $H^1(U;\Z_2)=0$. However, if $\kappa\in H^1(\Lambda_*;\Z_2)$ is the cohomology class of~\eqref{e:kappa}, Theorem~\ref{t:LS} implies that $j^*\kappa\neq 0$ in $H^1(U;\Z_2)$, and gives a contradiction.
\end{proof}

\bibliography{_biblio}
\bibliographystyle{amsalpha}

\end{document}